\documentclass[12pt,reqno,a4paper]{amsart}

\usepackage{amssymb, amsfonts, amsthm, amsmath, verbatim, graphicx,enumerate}

\usepackage{url}

\newcommand{\Z}{\mathbb{Z}}
\newcommand{\zz}{\mathbb{Z}_2}

\DeclareMathOperator{\ad}{ad}

\renewcommand \Im{\mathrm{Im}\,}
\newcommand{\g}{\mathfrak{g}}

\newcommand{\h}{\mathfrak{h}}
\newcommand{\m}{\mathfrak{m}}
\newcommand{\Span}{\mathrm{Span}}
\newcommand{\D}{\mathcal{D}}

\makeatletter
\newtheorem*{rep@theorem}{\rep@title}
\newcommand{\newreptheorem}[2]{%
\newenvironment{rep#1}[1]{%
\def\rep@title{#2 \ref{##1}}%
\begin{rep@theorem}}%
{\end{rep@theorem}}}
\makeatother

\theoremstyle{plain}
\newtheorem{theorem}{Theorem}
\newtheorem{proposition}{Proposition}
\newtheorem{lemma}{Lemma}
\newtheorem{corollary}{Corollary}
\newtheorem*{theorem*}{Theorem}
\newtheorem*{corollary*}{Corollary}
\newreptheorem{theorem}{Theorem}

\theoremstyle{definition}

\newtheorem*{definition*}{Definition}

\theoremstyle{remark}
\newtheorem{remark}{Remark}
\newtheorem{example}{Example}
\newtheorem*{example*}{Example}

\newtheorem*{exercise*}{Exercise}

\begin{document}

\title[Cohomology of $\mathbb{N}$-graded Lie algebras of maximal class over $\Z_2$]{Cohomology of $\mathbb{N}$-graded Lie algebras of maximal class over $\Z_2$}

\author{Yuri Nikolayevsky}
\author{Ioannis Tsartsaflis}

\address{Department of Mathematics and Statistics, La Trobe University, Melbourne, Australia 3086}
\email{Y.Nikolayevsky@latrobe.edu.au}
\email{itsartsaflis@students.latrobe.edu.au}




\thanks{The first named author was partially supported by ARC Discovery grant DP130103485}


\begin{abstract}
We compute the cohomology with trivial coefficients of Lie algebras $\m_0$ and $\m_2$ of maximal class over the field $\Z_2$. In the infinite-dimensional case, we show that the cohomology rings $H^*(\m_0)$ and $H^*(\m_2)$ are isomorphic, in contrast with the case of the ground field of characteristic zero, and we obtain a complete description of them. In the finite-dimensional case, we find the first three Betti numbers of $\m_0(n)$ and $\m_2(n)$ over $\zz$.
\end{abstract}

\maketitle


\section{Introduction}
\label{s:intro}

A Lie algebra $\g$ is said to be \emph{$\mathbb{N}$-graded}, if it is the direct sum of subspaces $\g_i, \; i \in \mathbb{N}$ (the \emph{homogeneous components}), such that $[\g_i, \g_j] \subset \g_{i+j}$. Obviously, finite-dimensional $\mathbb{N}$-graded Lie algebras are necessarily nilpotent. A great deal of attention in the literature has been focused on $\mathbb{N}$-graded Lie algebras for which the homogeneous components $\g_i$ are ``the smallest possible", that is, all of dimension one or, in the finite-dimensional case, $\dim \g_i =1$, for $i \le n:=\dim \g$, and $\g_i =0$, for $i > n$. With the additional condition that $\g$ is generated as an algebra by elements $e_1$ and $e_2$, spanning $\g_1$ and $\g_2$ respectively, one obtains that the subspaces $C_0=\g, \; C_k = \oplus_{i=k+2}^\infty \g_i, \; k >0$, are the terms of the central descending series. This defines the $\mathbb{N}$-graded filiform Lie algebras in the finite-dimensional case \cite{Vergne} and the $\mathbb{N}$-graded Lie algebras of maximal class \cite{SZ} (also called \emph{narrow} algebras). In characteristic zero, these algebras have been completely classified. In the infinite-dimensional case, one gets just three algebras \cite{Fial}, and independently \cite[Theorem~7.1]{SZ}. We list them here with their presentations:
\begin{align}\label{eq:m0def}
    & \m_0 = \Span (e_1, e_2, \dots), & \quad & [e_1, e_i] = e_{i+1}, \; i > 1, \\
    & \m_2 = \Span (e_1, e_2, \dots), & \quad & [e_1, e_i] = e_{i+1}, \; i > 1, \quad [e_2, e_j] = e_{j+2}, \; j > 2, \label{eq:m2def}\\
    & \mathcal{V} = \Span (e_1, e_2, \dots), & \quad & [e_i, e_j] = (j-i) e_{i+j}, \; i, j \ge 1.  \label{eq:vdef}
\end{align}
In the finite-dimensional case in characteristic zero, the classification of finite-dimensional $\mathbb{N}$-graded filiform Lie algebras was established in \cite{million}: one obtains the ``truncations" of the above three algebras, in particular,
\begin{align}\label{eq:m0ndef}
    & \m_0(n) = \Span (e_1, \dots, e_n), \; [e_1, e_i] = e_{i+1}, \; 1 < i < n, \\
    & \m_2(n) = \Span (e_1, \dots, e_n), \; [e_1, e_i] = e_{i+1}, \; 1 < i < n, \; [e_2, e_j] = e_{j+2}, \; 2 < j < n - 1 , \label{eq:m2ndef}
\end{align}
and $\mathcal{V}(n)$, plus another three infinite series, and five one-parameter families of low-dimensional algebras. The picture is more complicated in positive characteristic: by \cite{max1}, there are uncountably many isomorphism classes of Lie algebras of maximal class; the construction of all such algebras in odd characteristic is given in \cite{max2}, and in characteristic two, in \cite{max3}, with $\m_0$ and $\m_2$ being the simplest possible cases.


The cohomology of $\mathbb{N}$-graded Lie algebras of maximal class has been studied extensively over a field of characteristic zero \cite{Fial, FM, Vergne}, and at present is well-understood. In \cite{FM}, Fialowski and Millionschikov gave a full description of the cohomology with trivial coefficients of the algebras $\m_0$ and $\m_2$; the Betti numbers of $\mathcal{V}$ are found in \cite{goncharova}. In the finite-dimensional case, the cohomology of $\m_0(n)$ were found in \cite{Bord} (see also \cite{AS} and \cite{FM}). However, already for $\m_2(n)$ over a field of characteristic zero, our present knowledge is limited to the first two Betti numbers \cite{million, Vergne}.

The study of the cohomology of Lie algebras of maximal class over fields of positive characteristic is much less developed. The cohomology of the Heisenberg algebra is found in \cite{CJ, skol}. A recent result by Tsartsarflis \cite{Ts} states that over a field of characteristic two, the algebras $\m_0(n)$ and $\m_2(n)$ have the same Betti numbers (in contrast with the case of characteristic zero), and furthermore, every algebra of the so called \emph{Vergne class} admits a dual, non-isomorphic algebra, with the same Betti numbers.


In this paper we study the cohomology with trivial coefficients of the Lie algebras $\m_0$ and $\m_2$, and their finite dimensional truncations, $\m_0(n)$ and $\m_2(n)$, over the field $\Z_2$. Let $V = \Span(e_1, e_2, \dots)$ and let $\{e^i\}$ be the dual basis for $V^*$. Define the operator $D_1$ on $V^*$ by $D_1e^1=D_1e^2=0, \; D_1e^i = e^{i-1}$, for $i > 2$, and extend it to $\Lambda(V)$ as a derivation. For $\omega \in \Lambda(V)$ and $e^i \in V^*$, define $F(\omega,e^i)=\sum_{l=0}^\infty (D_1^l \omega) \wedge e^{i+l+1}$ (note that the sum on the right-hand side is finite).

Our main result in the infinite-dimensional case is as follows.

\begin{theorem}\label{t:main}
The cohomology rings $H^*(\m_0)$ and $H^*(\m_2)$ over the field $\Z_2$ are isomorphic. The respective cohomology classes of the cocycles
\begin{equation} \label{eq:classes}
e^1, \; e^2, \; F(e^{i_1}\wedge e^{ i_2}\wedge  \dots\wedge  e^{i_q},e^{i_q}),
\end{equation}
where $q \ge 1, \; 2\le i_1 {<}i_2{<}{\dots} {<}i_q$, form a basis for $H^*(\m_0)$ and for $H^*(\m_2)$, respectively.
\end{theorem}
Note that $H^*(\m_0)$ over $\zz$ is ``the same" as over a field of characteristic zero (compare with \cite[Theorem~3.4]{FM}). In contrast, the fact that $H^*(\m_0)$ and $H^*(\m_2)$ over $\zz$ are isomorphic (note that $\m_0$ and $\m_2$ are not isomorphic over any ground field) is specific to the $\zz$ case: over a field of characteristic zero, $H^*(\m_2)$ is very different \cite[Theorem~5.5]{FM}.

In the finite-dimensional case, which appears to be substantially harder that the infinite-dimensional one, we compute the first three Betti numbers of $\m_0(n)$ and the corresponding bases for $H^i(\m_0(n)), \; i=1,2,3$.

\begin{theorem}\label{t:b3m0z2}
    The first three Betti numbers of the Lie algebra $\m_0(n)$ over $\Z_2$ are given by
    \begin{enumerate}[{\rm (a)}]
    \item \label{it:b1} $b_1(\m_0(n))=2$,
    \item \label{it:b2} $b_2(\m_0(n))=\lfloor\frac12(n+1)\rfloor$, where $\lfloor.\rfloor$ denotes the integer part,
    \item \label{it:b3} $b_3(\m_0(n)) = \tfrac13 (2^p-1)(2^{p-1}-1) + \tfrac12 m(m-1) + \lfloor\tfrac12 (n-1)\rfloor$,
     where $n=2^p+m$ and $0<m \le 2^p$. 
   \end{enumerate}
\end{theorem}
An explicit form of the basis for $H^3(\m_0(n))$ is given in Theorem~\ref{t:3cocyclesm0n} of Section~\ref{s:m0}. Theorem~\ref{t:b3m0z2} also gives us the first three Betti numbers of $\m_2(n)$ (Corollary~\ref{cor:m2n} of Section~\ref{s:m2}), which in characteristic two are simply the same as those for $\m_0(n)$, by \cite[Theorem~1]{Ts}.


The paper is organised as follows. We begin with some short preliminaries in Section~\ref{s:pre}. We treat the algebras $\m_0$ and $\m_0(n)$ in Section~\ref{s:m0}. Parts \eqref{it:b1} and \eqref{it:b2} of Theorem~\ref{t:b3m0z2} follow from Proposition~\ref{p:H1H2m0n}. After some technical preparation similar to the arguments of \cite{FM}, we prove Theorem~\ref{t:main_H_m_0}, which is ``the $\m_0$-part" of Theorem~\ref{t:main}. We then proceed to the proof of Theorem~\ref{t:b3m0z2}\eqref{it:b3}.  This is the longest and most technically involved part of the paper. Finally, in Section~\ref{s:m2} we use a construction similar to \cite{Ts} to establish the isomorphism between $H^*(\m_0)$ and $H^*(\m_2)$, hence completing the proof of Theorem~\ref{t:main}.

\section{Preliminaries}
\label{s:pre}

Given a Lie algebra $\g$ over $\Z_2$ with a basis elements $e_i$, we denote the dual basis elements $e^i$. For convenience, we set $e^0=0$. For simplicity we write a monomial $q$-form $e^{i_1} \wedge e^{i_2} \wedge \dots \wedge e^{i_q} \in \Lambda^q(\g)$ as $e^{i_1i_2 \dots i_q}$. For a monomial $e^{i_1 i_2 \dots i_q}$, its \emph{degree} is defined to be $\sum_{j=1}^q i_j$. The \emph{homogeneous component} $\Lambda^q_k(\g)$ of degree $k$ and of rank $q$ is the span of all the monomials of degree $k$ and of rank $q$. We set $\Lambda_k(\g):=\oplus_q \Lambda^q_k(\g)$.

As usual, the \emph{differential} $d$ is defined by $d\xi(X,Y)=\xi[X,Y]$ for one-forms $\xi$, where $X,Y \in \g$, and then is extended to the exterior algebra $\Lambda(\g)$ as a derivation (so that $d(\omega_1 \wedge \omega_2)=d(\omega_1) \wedge \omega_2 + \omega_1 \wedge d(\omega_2)$). Then $d^2=0$ and one define the \emph{$q$-th cohomology group $H^q(\g)$} (with trivial coefficients) by $H^q(\g)=\ker(d:\Lambda^q \to \Lambda^{q+1})/\Im(d:\Lambda^{q-1} \to \Lambda^q)$. Then $H^q(\g)$ is a linear space over $\Z_2$; if its dimension is finite, it is called the \emph{$q$-th Betti number} $b_q(\g)$. It is immediate from the definition that if $\dim \g = n$, then
\begin{equation}\label{eq:dimker}
    b_q(\g) = \dim \ker(d:\Lambda^q \to \Lambda^{q+1}) + \dim \ker(d:\Lambda^{q-1} \to \Lambda^q) - \binom{n}{q-1},
\end{equation}
so to compute the Betti numbers it suffices to know the dimensions of the kernels of $d$ on the $\Lambda^q$'s. Also note that in the graded case (in particular, for the bases $\{e_i\}$ from (\ref{eq:m0def} -- \ref{eq:m2ndef})), the operator $d$ maps $\Lambda^q_k(\g)$ to $\Lambda^{q+1}_k(\g)$, and so $H^q(\g)$ is spanned by the classes of homogeneous elements; we get a decomposition (a bi-gradation) $H^q(\g)=\oplus_k H^q_k(\g)$. The multiplicative structure in $H(\g):=\oplus_q H^q(\g) $ is inherited from the wedge product.


\section{Cohomology of $\m_0$}
\label{s:m0}

In this section, we compute the cohomology of the infinite-dimensional Lie algebra $\m_0$ and also the first three Betti numbers of the finite-dimensional Lie algebras $\m_0(n)$ defined as follows (\ref{eq:m0def}, \ref{eq:m0ndef}):
\begin{gather*}
    \m_0=\Span(e_1,e_2,e_3, \dots), \quad [e_1,e_i]=e_{i+1}, \; \text{ for } i \ge 2, \\
    \m_0(n)=\Span(e_1,e_2,e_3, \dots, e_n), \quad [e_1,e_i]=e_{i+1}, \; \text{ for } 2 \le i \le n-1.
\end{gather*}

In the first few paragraphs, we closely follow the approach and the results of \cite[Section~3]{FM}, adapting them to the case of the ground field $\Z_2$. In effect, the outcome is that in the infinite-dimensional case, for $\g=\m_0$, the cohomology is ``the same" as that for a field of characteristic zero, while in the finite-dimensional case, for $\g=\m_0(n)$, the situation is more delicate -- not only the Betti numbers are different, but also the methods of \cite{FM, AS} and the very elegant approach of \cite[Appendix~B]{Bord} do not work directly.

For a monomial $e^{i_1 i_2 \dots i_q} \in \Lambda^q(\g), \; q \ge 1, \; i_1, i_2, \dots, i_q \ge 1$, (for both $\g=\m_0$ and $\g=\m_0(n)$) we have
\begin{equation}\label{eq:dm0}
\begin{split}
    d(e^{i_1 i_2 \dots i_q}) &=e^{1(i_1-1) i_2 \dots i_q}+e^{1i_1 (i_2-1) \dots i_q}+ \dots + e^{1i_1 i_2 \dots (i_q-1)}\\
    &=e^1 \wedge \big(e^{(i_1-1) i_2 \dots i_q}+e^{i_1 (i_2-1) \dots i_q}+ \dots + e^{i_1 i_2 \dots (i_q-1)}\big).
\end{split}
\end{equation}
It follows from \eqref{eq:dm0} that the subspaces $\Lambda_k(\g)$ are $d$-invariant.

Moreover, for any $\omega \in \Lambda(\g)$ we have $d(e^1 \wedge \omega)=0$ and $d(\omega) \in e^1 \wedge \Lambda(\g)$. Set $\h:=\Span(e_2, e_3, \dots)$ for $\m_0$, and $\h:=\Span(e_2, e_3, \dots, e_n)$ for $\m_0(n)$. Then $\h$ is abelian and from \eqref{eq:dm0} it follows that there is a well-defined linear operator $D$ on $\Lambda(\h)$ such that for $\omega \in \Lambda(\h)$, we have
\begin{equation}\label{eq:defD}
d\omega=e^1 \wedge (D\omega).
\end{equation}
It is easy to see that
\begin{equation}\label{eq:DLh}
    De^2=0, \; De^i = e^{i-1} \text{ for } i >2, \qquad D(\xi \wedge \eta)= D(\xi) \wedge \eta + \xi \wedge D(\eta) \text{ for } \xi, \eta \in \Lambda(\h),
\end{equation}
so $D$ is a derivation of $\Lambda(\h)$. Recall that the \emph{Lie derivative} with respect to $e_1$ is defined by taking the operator $(\ad_{e_1})^*$ on $\g^*$ to be the dual to $\ad_{e_1}$ on $\g$, and then extending it as a derivation to $\Lambda(\g)$.  Note that $D$ is just the restriction of $(\ad_{e_1})^*$ to $\Lambda(\h)$. Furthermore, $D(\Lambda^q_k(\h)) \subset \Lambda^q_{k-1}(\h)$, so that $D$ is ``nilpotent": for any $\omega \in \Lambda(\h)$ there exists $N=N(\omega) \ge 0$ such that $D^N \omega = 0$. For convenience, we define $D^0$ to be the identity map.

Since from \eqref{eq:dm0}, $\ker d = e^1 \wedge \Lambda(\h) \oplus \ker D$, to find the kernel of $d$ we need to find the kernel of $D$. This is given by the following lemma.

\begin{lemma}\label{l:kerD}
\begin{enumerate}[{\rm (a)}]
  \item \label{it:kerD}
  Let $\g=\m_0$. For any $\omega \in \Lambda(\h)$ and $e^i \in \h$ define
  \begin{equation}\label{eqFomegaei}
    F(\omega,e^i)= \sum\nolimits_{l = 0}^\infty D^l \omega \wedge e^{i+1+l}= \sum\nolimits_{l = 0}^{N(\omega)-1} D^l \omega \wedge e^{i+1+l}.
  \end{equation}
  Then $F(\omega,e^i) \in \ker D$ for $\omega \wedge e^i=0$ and moreover, the elements
  \begin{equation}\label{eq:Dbasis}
  \begin{gathered}
    F(e^{i_1 i_2 \dots i_q},e^{i_q})=e^{i_1 i_2 \dots i_q i_{q}+1} + D e^{i_1 i_2 \dots i_q} \wedge e^{i_q+2} + \dots \in \Lambda^{q+1}_k(\h), \\
    \text{where } q \ge 1, \; 2 \le i_1 < i_2 < \dots < i_q, \; k=  i_q+1 +\sum\nolimits_{j=1}^q i_j ,
  \end{gathered}
  \end{equation}
  form a basis for the kernel of the restriction of $D$ to $\Lambda^{q+1}_k(\h)$; the kernel of the restriction of $D$ to $\h^*$ is spanned by $e^2$.

  \item \label{it:kerDn}
  Let $\g=\m_0(n)$, viewed as the subspace of $\m_0$ spanned by the first $n$ vectors. Then $\ker D$ is the intersection of $\ker D$ constructed in \eqref{it:kerD} for the case $\g=\m_0$ with $\m_0(n)$.
\end{enumerate}
\end{lemma}

Note that in the Introduction we used $D_1=(\ad_{e_1})^*$ rather than $D$ to define $F$. This yields the same object, since in \eqref{eq:classes}, $D$ only acts on elements of $\Lambda(\h)$ and $D$ is the restriction on $D_1$ to $\Lambda(\h)$. Notice however that Lemma \ref{l:kerD} concerns $\ker D$, which is different to $\ker D_1$.

\begin{proof}
\eqref{it:kerD} The fact that $F(\omega,e^i) \in \ker D$ follows immediately, as from \eqref{eq:DLh}, for any $\omega \in \Lambda(\h)$ and $e^i \in \h$ we have
\begin{align*}
    DF(\omega,e^i)&= D \Big(\sum\nolimits_{l = 0}^\infty D^l \omega \wedge e^{i+1+l}\Big)\\
 &= \sum\nolimits_{l = 0}^\infty D^{l+1} \omega \wedge e^{i+1+l}+ \sum\nolimits_{l = 0}^\infty D^l \omega \wedge e^{i+l}\\
 &= \sum\nolimits_{l = 1}^\infty D^{l} \omega \wedge e^{i+l}+ \sum\nolimits_{l = 0}^\infty D^l \omega \wedge e^{i+l}\\
 &=\omega \wedge e^i ,
\end{align*}
as we are working over $\Z_2$. Notice in passing that this also shows that $D$ is surjective.

The fact that the elements given by \eqref{eq:Dbasis} are linearly independent is also easy, as from among the monomials $e^{j_1 j_2 \dots j_q j_{q+1}}, \; 2 \le j_1 < j_2 < \dots < j_q < j_{q+1}$ which appear on the right-hand side of the expansion of $F(e^{i_1 i_2 \dots i_q},e^{i_q})$, there is exactly one with the property that $j_{q+1} = j_q+1$, namely the monomial $e^{i_1 i_2 \dots i_q i_{q}+1}$. The fact that they indeed span the kernel of the restriction of $D$ to $\Lambda^{q+1}_k(\h)$ follows from the same observation and from the dimension count. The elements $F(e^{i_1 i_2 \dots i_q},e^{i_q}) \in \Lambda^{q+1}_k(\h)$ with $q \ge 1, \; 2 \le i_1 < i_2 < \dots < i_q, \; i_q+1 +\sum\nolimits_{j=1}^q i_j=k,$ are in one-to-one correspondence with the elements $e^{j_1 j_2 \dots j_q j_{q}+1} \in \Lambda^{q+1}_k(\h)$ with $2 \le j_1 < j_2 < \dots < j_q$. On the other hand, consider the linear operator $A:\Lambda^{q+1}_k(\h) \to \Lambda^{q+1}_{k-1}(\h)$ defined on the monomials as follows: $Ae^{j_1 j_2 \dots j_q j_{q+1}}=e^{j_1 j_2 \dots j_q j_{q+1}-1}$. Then $A$ is surjective and its kernel is spanned by the monomials $e^{j_1 j_2 \dots j_q j_{q}+1}$, so every surjective linear operator from $\Lambda^{q+1}_k(\h)$ to $\Lambda^{q+1}_{k-1}(\h)$ (in particular, $D$) has a kernel of the same dimension.

\eqref{it:kerDn} easily follows from the fact that for the operator $D$ defined for $\g=\m_0$, the subspace $\Lambda(\h)$ defined for $\m_0(n)$ is $D$-invariant, and the restriction of $D$ to it is the operator $D$ defined for $\m_0(n)$.
\end{proof}

With Lemma~\ref{l:kerD} we can easily finish the computation of the cohomology for $\g=\m_0$; we obtain the same answer as in \cite[Theorem~3.4]{FM}:

\begin{theorem}\label{t:main_H_m_0}
The cohomology classes of the cocycles
\begin{equation} \label{eq:classesm0}
e^1, \; e^2, \; F(e^{i_1 i_2 \dots i_q},e^{i_q}),
\end{equation}
where $q \ge 1, \; 2\le i_1 {<}i_2{<}{\dots} {<}i_q$, form a basis for $H^*(\m_0)$ over the field $\Z_2$.
\end{theorem}

Furthermore, the dimensions of the homogeneous components of $H^*(\m_0)$ over $\Z_2$ are the same as those over a field of characteristic zero, so in particular,
\begin{equation*}
\dim H_{k{+}\frac{q(q{+}1)}{2}}^q(\m_0)= P_q(k)-P_q(k-1),
\end{equation*}
where $P_q(k)$ is the number of partitions of a positive integer $k$ into $q$ parts. The products of the basis elements also have ``the same" decomposition as in \cite[Equation~(8)]{FM}, after reducing the coefficients modulo $2$.

\begin{proof}[Proof of Theorem~\ref{t:main_H_m_0}]
From Lemma~\ref{l:kerD}\eqref{it:kerD} we know $\ker D$, and so we know $\ker d = e^1 \wedge \Lambda(\h) \oplus \ker D$. The image of $d$ is just $e^1 \wedge \Lambda(\h)$, by \eqref{eq:defD} and from the surjectivity of $D$ (which has been established in the proof of Lemma~\ref{l:kerD}\eqref{it:kerD}). Putting these two facts together we get the claim.
\end{proof}


We now turn our attention to the case $\g=\m_0(n)$. We view $\m_0(n)$ as a subspace of $\m_0$ spanned by the first $n$ basis elements and for convenience, denote the operator $D$ defined for $\m_0$ by $\D$. The following Proposition easily follows from Lemma~\ref{l:kerD}.

\begin{proposition}\label{p:H1H2m0n}
The space $H^1(\m_0(n))$ is spanned by the classes of the elements $e^1, e^2$ and so $b_1(\m_0(n))=2$. The space $H^2(\m_0(n))$ is spanned by the classes of the elements $e^{1n}, \; F(e^i,e^i)=e^{i,i+1}+e^{i-1,i+3} + \dots + e^{2, 2i-1}, \; 2 \le i \le \tfrac12 (n+1)$, and so $b_2(\m_0(n))=\lfloor\frac12(n+1)\rfloor$.
\end{proposition}
\begin{proof}
The claim for $H^1(\m_0(n))$ is clear. For the second cohomology, by Lemma~\ref{l:kerD}\eqref{it:kerD}, the kernel of $\D$ is spanned by the elements $F(e^i,e^i)=e^{i,i+1}+e^{i-1,i+3} + \dots + e^{2, 2i-1}$. Since a sum of some number of the $F(e^i,e^i)$ belongs to $\m_0(n)$ if and only if each of them does (no two monomials of the different $F(e^i,e^i)$ may possibly cancel), we get by Lemma~\ref{l:kerD}\eqref{it:kerDn}:
\begin{equation}\label{eq:kerD2}
\ker D = \Span (F(e^i,e^i) \, : \, 2 \le i \le \tfrac12 (n+1)).
\end{equation}
Then $\ker d = e^1 \wedge \Lambda^1(\h) \oplus \ker D$ and so the second coboundary space is spanned by $e^{1i}, F(e^i,e^i),\; i=2, \dots, n-1$. Then, as the image of $d$ on the space of one-forms is spanned by $e^1 \wedge e^i$, for $1\le i\le n-1$,  the claim follows. \end{proof}

Proposition \ref{p:H1H2m0n} establishes  parts \eqref{it:b1} and \eqref{it:b2} of Theorem~\ref{t:b3m0z2}. The first two Betti numbers of $\m_0(n)$ over $\Z_2$ are the same as those over a field of characteristic zero \cite{AS}, but $b_3$ is different, as Theorem~\ref{t:b3m0z2}\eqref{it:b3} shows.

\begin{remark}\label{rem:tabtooth}
Explicitly, for small values of $n$, Theorem~\ref{t:b3m0z2}\eqref{it:b3} gives:
\begin{table}[h!]
\begin{tabular}{|c|c|c|c|c|c|c|c|c|c|c|c|c|c|c|c|c|c|c|}
  \hline
  $n$ & 3 & 4 & 5 & 6 & 7 & 8 & 9 & 10 & 11 & 12 & 13 & 14 & 15 & 16 & 17 & 18 & 19 & 20 \\
  \hline
  $b_3(\m_0(n))$ & 1 & 2 & 3 & 4 & 7 & 10 & 11 & 12 & 15 & 18 & 23 & 28 & 35 & 42 & 43 & 44 & 47 & 50 \\
  \hline
\end{tabular}
\end{table}

The sequence $b_3(\m_0(n))$ is the sequence A266540 in \cite{OEIS}\footnote{The authors are thankful to Omar E. Pol for pointing this out.}. To see that, we note that by the formula given in Theorem~\ref{t:b3m0z2}\eqref{it:b3}, $b_3(\m_0(n))=\frac12(b_3(\m_0(n-1))+b_3(\m_0(n+1)))$, for odd $n \ge 3$, and so it suffices to show that the even terms of the two sequences coincide, which is equivalent to the fact that the sequence $A_l:=\frac12 b_3(\m_0(2l))=\frac13(2^{2p-2}-1)+ s^2$, where $l = 2^{p-1}+s, \, 0 < s \le 2^{p-1}$, coincides with A256249. This is equivalent to the fact that $A_l$ is the $(l-1)$-st partial sum of the sequence A006257 given by $a_j=2(j-2^{\lfloor \log_2{j}\rfloor})+1$. But the latter partial sum equals $l^2-1-2(2^{p-1} s + \sum_{i=0}^{p-2} 2^{2i})$, and the claim follows.
\end{remark}

The proof of Theorem~\ref{t:b3m0z2}\eqref{it:b3} is based on the following Proposition. For brevity, let us denote the vector space $\Lambda^3(e_2, \dots, e_{n-1})$ by $W$. Denote $\h=\Span(e_2, \dots, e_{n})$.
\begin{proposition}\label{p:kerdm0n} 
For $m$ as defined in Theorem~\ref{t:b3m0z2}, there exists $\omega_k \in W$ for $2 \le k \le m$ such that
\begin{equation*}
\ker D_{|\Lambda^3(\h)} = \ker D_{|W} \oplus \Span(e^n \wedge F(e^k,e^k) + \omega_k \, : \, 2 \le k \le m).
\end{equation*}
\end{proposition}

We first prove the theorem assuming the Proposition.

\begin{proof}[Proof of Theorem~\ref{t:b3m0z2}\eqref{it:b3}]
For $n = 3$ the statement is easily verified: $H^3(\m_0(3))$ is spanned by the class of the single element $e^{123}$, so $b_1(\m_0(3))=1$, as claimed.

Assume $n \ge 4$. Denote $d_n$ the dimension of the kernel of the operator $D$ constructed for the algebra $\m_0(n)$. Then from Proposition~\ref{p:kerdm0n} we have $d_n=d_{n-1}+m-1$. It follows that for $n=2^p+m, \; 0<m \le 2^p$, we have $d_n=d_{2^p}+\frac12 m(m-1)$ and in particular,
\begin{equation}\label{E:rec}
d_{2^{p+1}}=d_{2^p}+ 2^{p-1}(2^p-1).
\end{equation}
We also have $d_4=1$, as for $\m_0(4)$ the space $\ker D$ is spanned by $e^{234}$. It follows from \eqref{E:rec} that $d_{2^p}=\tfrac13 (2^p-1)(2^{p-1}-1)$, and so $d_n=\tfrac13 (2^p-1)(2^{p-1}-1) + \tfrac12 m(m-1)$.

We have
\[
\dim \ker(d:\Lambda^3(\m_0(n)) \to \Lambda^{4}(\m_0(n))) = d_n+ \dim(e^1 \wedge \Lambda^{2}(\m_0(n))=d_n+\frac12 (n-1)(n-2).
\]
On the other hand, from Proposition~\ref{p:H1H2m0n},
\[
\dim \ker(d:\Lambda^{2}(\m_0(n) \to \Lambda^3(\m_0(n)) = (n-2)+\lfloor\tfrac12 (n+1)\rfloor,
\]
 and so the claim follows from \eqref{eq:dimker}.
\end{proof}

\begin{proof}[Proof of Proposition~\ref{p:kerdm0n}]
Any $\omega \in \Lambda^3(\h)$ can be uniquely represented as $\omega=e^n \wedge \xi + \omega'$, with $\xi \in \Lambda^2(e_2, \dots, e_{n-1}), \; \omega' \in \Lambda^3(e_2, \dots, e_{n-1})=W$. For $\omega$ to belong to $\ker D$ it is necessary that $D\xi = 0$ (so that $D\omega$ does not contain $e^n$). From the proof of Proposition~\ref{p:H1H2m0n} it follows that $\xi$ must be a linear combination of $F(e^k, e^k), \; k=2, \dots, \lfloor n/2\rfloor$. Extracting the homogeneous components we obtain that the proposition is equivalent to the following statement: for $2 \le k \le \lfloor n/2\rfloor$, there exists $\omega_k \in W$ such that $e^n \wedge F(e^k, e^k) + \omega_k \in \ker D$, if and only if $k \le m$.

The next step in the proof is the following lemma.

\begin{lemma}\label{l:matbino}
For $n \ge 4$ and $2 \le k \le \lfloor n/2\rfloor$, define $a=\lceil(n+2k+1)/3\rceil, \; b=\lfloor n/2\rfloor+k-1$.
There exists $\omega_k \in W$ such that $e^n \wedge F(e^k, e^k) + \omega_k  \in \ker D$ if and only if the linear system $Ax=(1,0,\dots,0)^t \in \Z_2^{k-1}$ has a solution $x \in \Z_2^{b-a+1}$, where $A$ is the $(k-1) \times (b-a+1)$-matrix given by
\begin{equation}\label{eq:A}
A_{ij}= \binom{n-(a+j-1)+2(i-1)}{(a+j-1)+(i-1)-k} \mod 2, \quad 1 \le i \le k-1, \; 1 \le j \le b-a+1,
\end{equation}
and as usual we set $\binom{N}{t}=0$ if $t < 0$ or $t> N$.
\end{lemma}


\begin{proof}
Suppose for some $\omega_k \in W$, the three-form $\omega=e^n \wedge F(e^k, e^k) + \omega_k $ belongs to $\ker D$ (where $2 \le k \le \lfloor n/2\rfloor$). Without loss of generality we can assume that $\omega_k$ is homogeneous, of the same degree as $e^n \wedge F(e^k, e^k)$, so that $\omega$ is homogeneous of degree $n+2k+1$.

By Lemma~\ref{l:kerD}, the form $\omega$ viewed as a three-form on $\m_0$, lies in the kernel of $\D$ and so is a linear combination of the forms $F(e^{s,r},e^r), \; 2 \le s < r$, where by homogeneity we can assume that $s+2r+1=n+2k+1$, from which it follows that $s=n+2k-2r$. Then $2 \le s \le r-1$ gives  $a \le r \le b$. Therefore for some $\mu_r \in \Z_2, \; r=a, \dots, b$ we have
\begin{align}
    \omega&= F(e^k, e^k) \wedge e^n + \omega_k =\sum\nolimits_{r=a}^b \mu_r F(e^{n+2k-2r,r},e^r) \notag\\
    &= \sum\nolimits_{r=a}^b \mu_r \sum\nolimits_{l=0}^\infty D^l(e^{n+2k-2r,r}) \wedge e^{l+r+1\notag }\\
    &= \sum\nolimits_{l=0}^\infty \sum\nolimits_{r=a}^b \mu_r D^l(e^{n+2k-2r,r}) \wedge e^{l+r+1}.\label{RHS}
\end{align}
As $n+2k-2r=s<r \le b$ and $b=\lfloor n/2\rfloor+k-1 \le 2\lfloor n/2\rfloor-1 < n$, no terms $D^l(e^{n+2k-2r,r})$ in the latter expression may possibly contain $e^N, \, N \ge n$. It follows that the only terms containing $e^N$ with $N \ge n$ in \eqref{RHS} are $\xi_N \wedge e^{N}$, where $\xi_N:=\sum\nolimits_{r=a}^{\min\{b,N-1\}} \mu_r D^{N-r-1} (e^{n+2k-2r,r})$. In fact, since $\omega\in \Lambda^3(\m_0(n))$, we have $\xi_N=0$ for all $N>n$ and equating the terms containing $e^n$ we get $\xi_n=F(e^k, e^k)$. Conversely, if $\xi_n=F(e^k, e^k)$, then $\xi_N=0$ for all $N>n$, as $\xi_{n+1}=D\xi_n=DF(e^k, e^k)=0, \; \xi_{n+2}=D^2\xi_n=D^2F(e^k, e^k)=0$, and so on. Thus a necessary and sufficient condition for the existence of $\omega_k \in W$ such that the three-form $\omega=e^n \wedge F(e^k, e^k) + \omega_k$ belongs to $\ker D$ is the existence of $\mu_r \in \Z_2, \; r=a, \dots, b$ such that
\begin{equation}\label{EF}
F(e^k, e^k)=\xi_n = \sum\nolimits_{r=a}^b \mu_r D^{n-r-1} (e^{n+2k-2r,r}).
\end{equation}
(the summation on the right-hand side is up to $b$ as $b \le n-1$). Note that both sides are homogeneous two-forms of degree $2k+1$. Recall that $F(e^k, e^k)= e^{k,k+1}+ e^{k-1,k+2} + \dots + e^{2,2k-1}$, and observe that
\[
D^{n-r-1} (e^{n+2k-2r,r})=\sum\nolimits_{i =0}^{n-r-1} \tbinom{n-r-1}{i} e^{2k-r+i+1,r-i}.
\]
So expanding and equating coefficients of the corresponding monomials we see that \eqref{EF} is equivalent to the following system:
\begin{align*}
    \sum\nolimits_{r=a}^b \mu_r \big(\tbinom{n-r-1}{r-k}+\tbinom{n-r-1}{r-(k+1)}\big) &=1 \mod 2, \\
    \sum\nolimits_{r=a}^b \mu_r \big(\tbinom{n-r-1}{r-(k-1)}+\tbinom{n-r-1}{r-(k+2)}\big)&=1 \mod 2, \\
    \vdots & \\
    \sum\nolimits_{r=a}^b \mu_r \big(\tbinom{n-r-1}{r-2}+\tbinom{n-r-1}{r-(2k-1)}\big)&=1 \mod 2.
    \end{align*}
Now the linear combination of the first $s \le k-1$ of the above equations with the coefficients $\binom{2s-1}{s-1}, \binom{2s-1}{s-2}, \dots, \binom{2s-1}{1}, \binom{2s-1}{0}$ respectively gives
\[
\sum\nolimits_{r=a}^b \mu_r \bigg(\sum\nolimits_{i=0}^{2s-1}\binom{2s-1}{i} \binom{n-r-1}{r-k-s+i}\bigg)= \sum\nolimits_{r=a}^b \mu_r \binom{n-r+2s-2}{r-k+s-1}
\]
 on the left-hand side (as $\sum_{i=0}^l\binom{l}{i}\binom{N}{t+i}=\sum_{i=0}^l\binom{l}{l-i}\binom{N}{t+i}=\binom{N+l}{t+l}$ by Vandermonde's identity). On the right-hand side we obtain $\binom{2s-1}{s-1}+ \binom{2s-1}{s-2} + \dots + \binom{2s-1}{1} + \binom{2s-1}{0} = \frac12 \times 2^{2s-1}=2^{2s-2}$, which is odd when $s=1$ and even otherwise. Thus the above system of equations is equivalent to the following one:
\begin{equation*}
    \sum\nolimits_{r=a}^b \mu_r \tbinom{n-r}{r-k}=1 \mod 2, \quad \sum\nolimits_{r=a}^b \mu_r \tbinom{n-r+2s-2}{r-k+s-1}=0 \mod 2, \text{ for } 2 \le s \le k-1.
\end{equation*}
This is equivalent to the claim of the lemma if we define $x=(\mu_a, \mu_{a+1}, \dots, \mu_b)^t$.
\end{proof}

In order to use Lemma~\ref{l:matbino} to conclude the proof of the proposition,  we need to show that the system $Ax=(1,0,\dots,0)^t$ has a solution if and only if $k \le m$. Even though we are working over $\zz$, let us say that vectors $x,y$ are \emph{orthogonal} if $x^ty=0$.



To prove the \textbf{necessity} we show that, assuming $k > m$, the first row of $A$ belongs to the span of the next $m-1$ rows, namely that
\begin{equation}\label{eq:lindeprows}
\big(\tbinom{k-m-1}{0}, \tbinom{k-m-1}{1}, \dots, \tbinom{k-m-1}{k-m-1}, 0, \dots, 0 \big) A = 0 \mod 2.
\end{equation}
Then any $x$ orthogonal to all the rows of $A$ starting from the second one, must also be orthogonal to the first row, and so the system $Ax=(1,0,\dots,0)^t$ has no solutions. To establish \eqref{eq:lindeprows} we need to show that for every $j=1, \dots, b-a+1$, we have
\begin{equation*}
    \sum\nolimits_{i=1}^{k-m} \binom{k-m-1}{i-1}\binom{n-(a+j-1)+2(i-1)}{(a+j-1)+(i-1)-k}=0 \mod 2.
\end{equation*}
which is equivalent (by substitution $r=a+j-1, \; l=i-1, \; N=k-m-1, \; n=2^p+m$) to showing that for all $r=a, \dots, b$,
\begin{equation}\label{eq:bibi}
    \sum\nolimits_{l=0}^{N} \binom{N}{l}\binom{2^p-1-(r-k+N-2l)}{r-k+l}=0 \mod 2.
\end{equation}
We require the following Lemma.

\begin{lemma}\label{l:binocong}
Suppose $p \ge 2$ and let $x, y \in \Z$.
\begin{enumerate}[{\rm (a)}]
  \item \label{it:binocong2}
  If $0 \le x < y < 2^p$, then $\binom{2^p+x}{y} = 0 \mod 2$. 

  \item \label{it:binocong1}
If $x, y \le 2^p-2$ and $y, x+y > 0$, then $\binom{2^p-1-x}{y} = \binom{y+x}{y} \mod 2$.

\end{enumerate}

\end{lemma}
\begin{proof}
By Kummer's Theorem, a binomial coefficient $\binom{q}{t}$ with $0 \le t$ is odd if and only if there is a place in the binary representation where $q$ has $0$ and $t$ has $1$ and, when $0 \le t \le q$, if and only if there is a place in the binary representation where both $q-t$ and $t$ have $1$.

(a) For $\binom{2^p+x}{y} = 1 \mod 2$, the binary representation of $2^p+x$ must have a $1$ at all the places where the binary representation of $y$ does. But as $y < 2^p$, this implies that the binary representation of $x$ has a $1$ at all the places where the binary representation of $y$ does, which contradicts the fact that $y>x$.

(b) First suppose $x \ge 0$. Then $\binom{2^p-1-x}{y}$ is even if and only if there is a place in the binary representation where $2^p-1-x$ has $0$ and $y$ has $1$ if and only if there is a place in the binary representation where $x$ has $1$ and $y$ has $1$ if and only if $\binom{y+x}{y}$ is even.

Now let $x < 0$. So $\binom{y+x}{y}=0$. Denote $z=-x-1 \ge 0$. Then $\binom{2^p-1-x}{y}=\binom{2^p+z}{y}$ and $0 \le z<y \le 2^p-2$ by our assumption. By part (a), $\binom{2^p+z}{y}=\binom{z}{y} \mod 2$, and $\binom{z}{y} =0$ as $z < y$. So $\binom{y+x}{y}=\binom{2^p+z}{y} \mod 2$.
\end{proof}

To apply Lemma~\ref{l:binocong}\eqref{it:binocong1} to the binomial coefficients $\binom{2^p-1-(r-k+N-2l)}{r-k+l}$ from \eqref{eq:bibi} we need to check few inequalities. We have $r-k \ge a-k=\left\lceil \frac13(n-k+1)\right\rceil \ge \left\lceil \frac13(n-\left \lfloor \frac12 n \right\rfloor +1) \right\rceil = \left\lceil  \frac13 (\left\lceil  \frac12 n \right\rceil+1) \right\rceil \ge 1$ and so $r-k+l \ge 1$ and $(r-k+l)+(r-k+N-2l) \ge 1$. Furthermore, $r-k+l,r-k+N-2l \le r-k+N \le b-k+N = \lfloor  \frac12 n \rfloor +N-1$, and $\lfloor \frac12 n \rfloor +N-1= \lfloor \frac12 n \rfloor  + k-m-2 \le 2 \lfloor \frac12 n \rfloor -m-2 = 2 \lfloor 2^{p-1} + \frac12 m \rfloor-m-2 \le 2^{p}-2$. So the hypotheses of Lemma~\ref{l:binocong}\eqref{it:binocong1} are satisfied with $x=r-k+N-2l, y=r-k+l$. So Lemma~\ref{l:binocong}\eqref{it:binocong1} gives $\binom{2^p-1-(r-k+N-2l)}{r-k+l}=\binom{2(r-k)+N-l}{r-k+l} \mod 2$, for every $l=0, \dots, N$. Vandermonde's identity gives $\binom{2(r-k)+N-l}{r-k+l}=\sum_{i=0}^{N-l} \binom{N-l}{i}\binom{2(r-k)}{r-k+l-i}$, and hence the left-hand side of \eqref{eq:bibi} is congruent modulo $2$ to
\begin{align*}
    \sum_{l=0}^{N}\binom{N}{l}&\sum_{i=0}^{N-l}\binom{N-l}{i}\binom{2(r-k)}{r-k+l-i}= \sum_{i,l \ge 0; i+l \le N}\binom{N}{i,l,N-l-i}\binom{2(r-k)}{r-k+l-i} \\
    &= \sum_{i> l \ge 0; i+l \le N} \binom{N}{i,l,N-l-i}\bigg(\binom{2(r-k)}{r-k+l-i}
    + \binom{2(r-k)}{r-k+i-l} \bigg) \\
    &\qquad+ \sum_{i \ge 0; 2i \le N}  \binom{N}{i,i,N-2i}\binom{2(r-k)}{r-k} \\
    &= 0 \mod 2,
\end{align*}
as $\binom{2(r-k)}{r-k+l-i} =\binom{2(r-k)}{r-k+i-l}$ and $\binom{2(r-k)}{r-k}=2 \binom{2(r-k)-1}{r-k}$. This completes the proof of necessity.


To prove the \textbf{sufficiency} we explicitly produce, for any $2 \le k \le m$, a vector $x \in \Z_2^{b-a+1}$ such that $Ax=(1,0, \dots, 0)^t \in \Z_2^{k-1}$:
\begin{equation}\label{eq:x}
    x_j=\sum\nolimits_{s=0}^{p-1} \binom{m-k}{n-(a+j-1)-2^s}, \quad j=1, \dots, b-a+1. 
\end{equation}
By Lemma~\ref{l:matbino} we need to show that for all $i=1, \dots, k-1$,
\begin{equation}\label{eq:Ax}
\sum_{j=1}^{b-a+1} \bigg(\binom{n-(a+j-1)+2(i-1)}{(a+j-1)+(i-1)-k} \sum_{s=0}^{p-1} \binom{m-k}{n-(a+j-1)-2^s}\bigg) \mod 2 = \delta_{1i}.
\end{equation}
We first show that the expression on the left-hand side of \eqref{eq:Ax} can be rewritten as
\begin{equation*} 
\sum_{s=0}^{p-1} \sum_{j \in \Z} \binom{n-(a+j-1)+2(i-1)}{(a+j-1)+(i-1)-k}  \binom{m-k}{n-(a+j-1)-2^s} \mod 2,
\end{equation*}
so that there is no contribution from the values $j \le 0$ and $j \ge b-a+1$. The latter is easy: for the first binomial coefficient to be nonzero we need to have $n-(a+j-1)+2(i-1) \ge (a+j-1)+(i-1)-k$ which gives $2j \le n+k+i+1-2a \le n+2k-2a$, as $i \le k-1$, so $j \le \lfloor n/2\rfloor +k-a=b-a+1$. To prove the former, we first look at the second binomial coefficient, from which we get $m-k \ge n-(a+j-1)-2^s$, so $ j \ge n-a+1+k-m-2^s \ge n-\frac13(n+2k+1)-\frac23+1+k-m-2^s= \frac13(2^{p+1}+k-m-3 \cdot 2^s)$. Now if $s < p-1$ the expression on the right-hand side is positive, as $m \le 2^p$, and we are done. Suppose $s=p-1$. Then we have $j \ge \frac13(2^{p-1}+k-m)$, which still implies $j > 0$ unless $m=2^{p-1}+k+l, \; l \ge 0$, in which case we have $j \ge - \frac13 l$. Then
\[
a=\left\lceil \frac{2^p+m+2k+1}3\right\rceil=\left\lceil \frac{2^p+2^{p-1}+3k+l+1}3\right\rceil=2^{p-1}+k+\left\lceil\frac{l+1}3 \right\rceil
\] 
and the first binomial coefficient has the form $\binom{2^p+x}{y}$, where
\begin{align*}
x&=n-(a+j-1)+2(i-1)-2^p=m-(a+j-1)+2(i-1)\\
&=2^{p-1}+k+l-(a+j-1)+2(i-1)=l+1-\left\lceil(l+1)/3 \right\rceil+2(i-1)-j,\\
y&=(a+j-1)+(i-1)-k=2^{p-1}+\left\lceil(l+1)/3 \right\rceil +j+i-2.
\end{align*}
Note that as $i\ge 1$, we have $x\ge 0$ if $j\le 0$. Also if $j\le 0$, then as $i\le k-1$, we have $y\le 2^{p-1}+\left\lceil(l+1)/3 \right\rceil +k-3\le 2^{p-1}+l +k-2=m-2< 2^p$. Moreover, if $j\le 0$, then as $i\le k-1$, we have $
y-x = (2^{p-1}+\left\lceil (l+1)/3 \right\rceil +j+i-2)-(l+1-\left\lceil (l+1)/3 \right\rceil+2(i-1)-j)=2^{p-1}+2\left\lceil (l+1)/3 \right\rceil -(l+1)+2j-i\ge 2^{p-1}+2-(l+1)-k=2^{p}-m+1>0$.
So the  hypotheses of Lemma~\ref{l:binocong}\eqref{it:binocong2} are satisfied, and hence the binomial coefficient $\binom{2^p+x}{y}$ is even. So it remains to establish that
\begin{equation}\label{eq:Ax1a}
\sum_{s=0}^{p-1} \sum_{j \in \Z} \binom{n-(a+j-1)+2(i-1)}{(a+j-1)+(i-1)-k}  \binom{m-k}{n-(a+j-1)-2^s}  \mod 2 = \delta_{1i},
\end{equation}
for all $i=1, \dots, k-1$.

A clear advantage of \eqref{eq:Ax1a} is that it ``takes care of itself" -- we do not have to worry about the limits. Changing the summation variable in \eqref{eq:Ax1a} to $h=n-(a+j-1)-2^s$ we obtain that \eqref{eq:Ax1a} is equivalent to
\begin{equation}\label{eq:Ax2}
\sum_{s=0}^{p-1} \sum_{h \in \Z} \binom{2^s+2(i-1)+h}{n-2^s+(i-1)-k-h}  \binom{m-k}{h} \mod 2 = \delta_{1i}.
\end{equation}
Now for a polynomial $P \in \Z_2[t]$ and $l \in \Z$ we denote $\{P\}_l$ the coefficient of $t^l$ in $P$. Consider the polynomial $P_{x,y}(t)=(t^2+t)^x(t^2+t+1)^y$. We have
\begin{equation*}
    P_{x,y}(t)=\sum_{h \in \Z}\binom{y}{h}(t^2+t)^{x+h}=\sum_{h,s \in \Z}\binom{y}{h}\binom{x+h}{s}t^{x+h+s}= \sum_{l \in \Z} \sum_{h \in \Z}\binom{x+h}{l-x-h}\binom{y}{h}t^l,
\end{equation*}
so the left-hand side of \eqref{eq:Ax2} equals
\begin{align*}
    \sum_{s=0}^{p-1} \{P_{2^s+2(i-1),m-k}\}_{n+3(i-1)-k} &= \bigg\{\sum_{s=0}^{p-1} (t^2+t)^{2^s+2(i-1)}(t^2+t+1)^{m-k}\bigg\}_{n+3(i-1)-k}\\
    &= \bigg\{\sum_{s=0}^{p-1} (t^2+t)^{2^s}(t^2+t)^{2(i-1)}(t^2+t+1)^{m-k}\bigg\}_{n+3(i-1)-k} \\
    &= \{ (t^{2^{p}}+t)(t^2+t)^{2(i-1)}(t^2+t+1)^{m-k}\}_{n+3(i-1)-k}
\end{align*}
modulo $2$ (since as $(t^2+t)^{2^s}=t^{2^{s+1}}+t^{2^s}$ in $\Z_2[t]$ and so $\sum_{s=0}^{p-1} (t^2+t)^{2^s}=t^{2^{p+1}}+t \mod 2$). Now, if in the expansion of the latter polynomial we take $t$ from the first parentheses, then the maximal degree of $t$ in the resulting terms will be $1+4(i-1)+2(m-k) \le 2m  -1 + 3(i-1)- k < n+3(i-1)-k$, as $i \le k-1$ and $n=2^p+m, \; m \le 2^p$. It follows that
\begin{align*}
    \sum_{s=0}^{p-1} \{P_{2^s+2(i-1),m-k}\}_{n+3(i-1)-k} & = \{ t^{2^p}(t^2+t)^{2(i-1)}(t^2+t+1)^{m-k}\}_{n+3(i-1)-k} \\
    &= \{ (t+1)^{2(i-1)}(t^2+t+1)^{m-k}\}_{m+(i-1)-k}\\
    &=\sum_{l\in\Z} \{ (t+1)^{2(i-1)}\}_{i-1+l} \{(t^2+t+1)^{m-k}\}_{m-k-l}\\
    &= \{ (t+1)^{2(i-1)}\}_{i-1} \{(t^2+t+1)^{m-k}\}_{m-k} \mod 2,
\end{align*}
where the last equality follows from the symmetry: for the polynomial $f(t)=(t+1)^{2(i-1)}$ we have $f(t)=t^{2(i-1)}f(t^{-1})$, so $\{(t+1)^{2(i-1)}\}_{i-1+l} =\{ (t+1)^{2(i-1)}\}_{i-1-l}$, and similarly $\{(t^2+t+1)^{m-k}\}_{m-k-l}=\{(t^2+t+1)^{m-k}\}_{m-k+l}$.

Now if $i > 1$ we obtain $\{ (t+1)^{2(i-1)}\}_{i-1}=\binom{2(i-1)}{i-1} = 0 \mod 2$, as required. If $i=1$ we get $\{(t^2+t+1)^{m-k}\}_{m-k}=\{\sum_{l}\binom{m-k}{l}(t^2+t)^l\}_{m-k}= \{\sum_{l,h}\binom{m-k}{l}\binom{l}{h}t^{h+l}\}_{m-k}= \sum_{l}\binom{m-k}{l}\binom{l}{m-k-l}=\sum_{s}\binom{m-k}{s}\binom{m-k-s}{s}$, where $s=m-k-l$. The terms with $s < 0$ vanish, and the term with $s=0$ is $1$. For $s > 0$, consider the first place, counting from the right, where the binary expansion of $s$ has a $1$. Then by Kummer's Theorem, for $\binom{m-k}{s}$ to be nonzero, the binary expansion of $m-k$ must have a $1$ at the same place, so the binary expansion of $m-k-s$ will have zero at that place, thus $\binom{m-k-s}{s}=0$. Hence $\{(t^2+t+1)^{m-k}\}_{m-k}=1 \mod 2$, as required. This concludes the proof of Proposition~\ref{p:kerdm0n} and hence of Theorem~\ref{t:b3m0z2}\eqref{it:b3}.
\end{proof}

Note that one can extract from the above proof an explicit basis for the space of three-cocycles of $\m_0(n)$ (and hence for $H^3(\m_0(n))$). We have the following theorem.

\begin{theorem}\label{t:3cocyclesm0n}
For $n \ge 4, \; n=2^p+m, \; 0<m \le 2^p$ and for $2 \le k \le m$, define the numbers $a=\lceil(n+2k+1)/3\rceil, \; b=\lfloor n/2\rfloor +k-1$. Let $B_n$ be the set of elements of $\m_0(n)$ of the form
\begin{equation*}
    \sum_{r=a}^b\sum_{s=0}^{p-1} \binom{m-k}{n-r-2^s} F(e^{n+2k-2r,r},e^r) \! = \!\! \sum_{r=a}^b\sum_{s=0}^{p-1} \binom{m-k}{n-r-2^s} \! \sum_{l \ge 0} D^l(e^{n+2k-2r} \wedge e^r) \wedge e^{r+l+1},
\end{equation*}
for $2 \le k \le m$, where $D$ is the linear operator defined by \eqref{eq:defD} and the binomial coefficients are taken modulo $2$. Then classes of the elements of the set
\begin{equation*}
    \{e^{1,i-1,i}, \quad 2+\lfloor n/2\rfloor \le i \le n\} \cup \bigcup_{4 \le t \le n} B_t.
\end{equation*}
is a basis for the cohomology space $H^3(\m_0(n)), \; n \ge 4$, over the field $\Z_2$.
\end{theorem}
\begin{proof}
We start with the elements $e^{1,i-1,i}, \; 2+\lfloor n/2\rfloor \le i \le n$. They are linearly independent cocycles and the space spanned by them has the correct dimension, which is the codimension of the space of coboundaries in the space spanned by $e^{1ij}, \; 1<i<j \le n$, by Proposition~\ref{p:H1H2m0n}. It suffices to show that neither of them is a coboundary. But if it were so, then by homogeneity we would have had that $e^{1,i-1,i}$ is the coboundary of a linear combination of the elements $e^{kl}, \; 2 \le k < l \le n, \; k+l=2i$, that is, of the elements $e^{i-k,i+k}, \; k=1, \dots, n-i$ (note that as $i \ge 2+\lfloor n/2\rfloor $, we have $2i-n-1 \ge 2$). But the coboundary of any such element is the sum of exactly two monomials, $e^{1,i-k-1,i+k}+e^{1,i-k,i+k-1}$, so the coboundary of any linear combination of them is a sum of an even number of monomials, hence cannot be equal to $e^{1,i-1,i}$.

As to the element from the sets $B_t$, no linear combination of them is a coboundary (as any coboundary is a multiple of $e^1$). Moreover, from  Proposition~\ref{p:kerdm0n} (both the statement and the proof) it follows that they form a basis for the kernel of $D$, where the form of the elements given in the statement follows from Lemma~\ref{l:matbino} and Equation~\eqref{eq:x}.
\end{proof}


\begin{example}\label{ex:3cocyclesm0n}
For $n = 4, \dots, 12$, the space of 3-cocycles of $\m_0(n)$ is spanned by the three-forms $e^{1ij}$, $1<i<j\le n$, and the three-forms from the following table in the rows labelled by the numbers less than or equal to $n$.
{\renewcommand{\arraystretch}{1.25}
\begin{equation*}
\begin{tabular}{|c|l|}
  \hline
    4 & $e^{234}$ \\ \hline
    5 &  \\ \hline
    6 & $e^{245}+e^{236}$ \\ \hline
    7 & $e^{345}+e^{246}+e^{237}, \; e^{356}+e^{257}+e^{347}$ \\ \hline
    8 & $e^{256}+e^{247}+e^{238}, \; e^{456}+e^{357}+e^{258}+e^{348}, e^{467}+e^{278}+e^{368}+e^{458}$ \\ \hline
    9 &  \\ \hline
    10 & $e^{267}+e^{258}+e^{249}+e^{23(10)}$ \\ \hline
    11 & \parbox{11cm}{\vskip 2pt $e^{367}+e^{268}+e^{358}+e^{349}+e^{24(10)}+e^{23(11)}, \\ e^{378}+e^{279}+e^{369}+e^{35(10)}+e^{25(11)}+e^{34(11)}$ \vskip 3pt }\\ \hline
    12 & \parbox{11cm}{\vskip 2pt $e^{467}+e^{368}+e^{458}+e^{269}+e^{25(10)}+e^{24(11)}+e^{23(12)}, \\ e^{478}+e^{289}+e^{379}+e^{469}+e^{45(10)}+e^{35(11)}+e^{25(12)}+e^{34(12)}, \\ e^{489}+e^{38(10)}+e^{47(10)}+e^{28(11)}+e^{46(11)}+e^{27(12)}+e^{36(12)}+e^{45(12)}$ \vskip 3pt } \\
  \hline
\end{tabular}
\end{equation*}
}
\end{example}

\section{Cohomology of $\m_2$}
\label{s:m2}

In this section, we compute the cohomology of the infinite-dimensional Lie algebra $\m_2$ given by \eqref{eq:m2def}:
\begin{equation*}
\m_2 = \Span (e_1, e_2, \dots),  \qquad  [e_1, e_i] = e_{i+1}, \; i > 1, \quad [e_2, e_j] = e_{j+2}, \; j > 2,
\end{equation*}
hence completing the proof of Theorem~\ref{t:main}. First we state the following result for the truncation $\m_2(n)$.

\begin{corollary}\label{cor:m2n}
The first three Betti numbers of the Lie algebra $\m_2(n), \; n \ge 5$, over $\Z_2$ are given by $b_1(\m_2(n))=2, \; b_2(\m_2(n))=[\tfrac12(n+1)]$, and
    \begin{equation*}
        b_3(\m_2(n)) = \tfrac13 (2^p-1)(2^{p-1}-1) + \tfrac12 m(m-1) + [\tfrac12 (n-1)],
    \end{equation*}
    where $n=2^p+m, \; 0<m \le 2^p$.
\end{corollary}
\begin{proof}
By \cite[Theorem~1]{Ts}, the Betti numbers of $\m_2(n)$ and of $\m_0(n)$ over $\zz$ are the same. The claim then follows from Theorem~\ref{t:b3m0z2}.
\end{proof}

\begin{remark} \label{rem:basm2n}
It is easy to see that $H^1(\m_2(n))$ is spanned by the cohomology classes of $e^1$ and $e^2$ and that $H^2(\m_2(n))$ is spanned by the cohomology classes of the elements $e^{1n}+e^{2,n-1}, \; e^{i,i+1}+e^{i-1,i+3} + \dots + e^{2, 2i-1}$, where $2 \le i \le \tfrac12 (n+1)$. A basis for $H^3(\m_2(n))$ can be found by applying the map $f$ from \cite[Definition~3]{Ts} (see below) to the elements of the basis for $H^3(\m_0(n))$ constructed in Theorem~\ref{t:3cocyclesm0n}; the resulting basis is the same.
\end{remark}

In the infinite-dimensional case, we follow the construction of \cite{Ts}. As in the Introduction, let $V = \Span(e_1, e_2, \dots)$, and define the operator $D_1$ on $V^*$ by $D_1e^1=D_1e^2=0$, $D_1e^i = e^{i-1}$, for $i > 2$, and then extend it to $\Lambda(V)$ as a derivation. Note that any $\omega \in \Lambda^q(V), \; q \ge 2$, has a unique presentation in the form $\omega = e^1 \wedge \xi + e^2 \wedge \eta + \zeta$, where $\xi \in \Lambda^{q-1}(e_2,e_3,\dots), \; \eta \in \Lambda^{q-1}(e_3,e_4,\dots)$ and $\zeta \in \Lambda^{q}(e_3,e_4,\dots)$. Note that $\xi, \eta$ and $\zeta$ linearly depend on $\omega$.

Define the linear map $f$ on $\Lambda(V)$ by setting $f(e^1 \wedge \xi + e^2 \wedge \eta + \zeta)=e^1 \wedge \xi + e^2 \wedge (\eta +D_1 \xi) + \zeta$ on the forms of rank at least two, and taking it to be the identity on $V^*$. The following properties of $f$ are easy to check:
\begin{itemize}
  \item $f$ is an involution, hence a bijection, and $f^{-1}=f$,
  \item the restriction of $f$ to $\Lambda(e_2,e_3,\dots)$ is the identity,
  \item $f$ preserves the homogeneous components: $f(\Lambda^q_k(V))=\Lambda^q_k(V)$.
\end{itemize}

The main feature of $f$ is the fact that it \emph{interweaves} the differentials of $\m_0$ and $\m_2$. More precisely, consider $\m_0$ and $\m_2$ to have the same underlying linear space $V$, but to be defined by the brackets \eqref{eq:m0def} and \eqref{eq:m2def} respectively relative to the same basis $\{e_1, e_2, \dots\}$ for $V$. Then for all $\omega \in \Lambda(V)$, we have
\begin{equation}\label{eq:interv}
f d_0 \omega = d_2 f \omega, \qquad f d_2 \omega = d_0 f \omega,
\end{equation}
where $d_0$ and $d_2$ are the differentials on $\m_0$ and $\m_2$ respectively. The first equation is easily verified for $\omega=e^i$, and the proof for $\omega \in \Lambda^q(V), \; q \ge 2$, is identical to the proof of \cite[Proposition~1]{Ts}. The second one follows, as $f$ is an involution.

\begin{proof}[Proof of Theorem~\ref{t:main}]
By \eqref{eq:interv}, $f$ bijectively maps cocycles and coboundaries of $\m_0$ to cocycles and coboundaries of $\m_2$ respectively. It follows that $H^*(\m_2)$ is spanned by the classes of the images under $f$ of the elements \eqref{eq:classesm0}. As $f$ acts on all those elements as the identity, we obtain that the basis for $H^*(\m_2)$ is the set of the classes of the same cocycles.

The fact that the multiplicative structure is preserved follows from the fact that the restriction of $f$ to $\Lambda(e_2,e_3,\dots)$ is the identity and that multiplication by $e^1$ is trivial in both $H^*(\m_0)$ and $H^*(\m_2)$. Multiplication by $e^1$ is trivial in $H^*(\m_0)$ because $e^1 \wedge \omega$ is a $d_0$-coboundary, for any $\omega$ (see the proof of Theorem~\ref{t:main_H_m_0}). To see that multiplication by $e^1$ is trivial in $H^*(\m_2)$, notice that for any $\omega$ in the list \eqref{eq:classesm0}, one has $D\omega=0$  (which is essentially assertion~\eqref{it:kerD} of Lemma~\ref{l:kerD}), and so $f(e^1 \wedge \omega)=e^1 \wedge \omega$, which is then a $d_2$-coboundary, as $f$ maps coboundaries to coboundaries.
\end{proof}

\vskip.5cm
\emph{Acknowledgements.} We are very grateful to Grant Cairns for drawing our attention to the question, for many useful discussions and for his help which improved the presentation of this paper.

	\bibliographystyle{amsplain}
	\bibliography{maxcnt21}

\def\cprime{$'$}
\providecommand{\bysame}{\leavevmode\hbox to3em{\hrulefill}\thinspace}
\providecommand{\MR}{\relax\ifhmode\unskip\space\fi MR }
\providecommand{\MRhref}[2]{%
  \href{http://www.ams.org/mathscinet-getitem?mr=#1}{#2}
}
\providecommand{\href}[2]{#2}
\begin{thebibliography}{10}

\bibitem{OEIS}
\emph{The {O}n-{L}ine {E}ncyclopedia of {I}nteger {S}equences}, published
  electronically at \url{http://oeis.org/}, 2015.

\bibitem{AS}
Grant~F. Armstrong and Stefan Sigg, \emph{On the cohomology of a class of
  nilpotent {L}ie algebras}, Bull. Austral. Math. Soc. \textbf{54} (1996),
  no.~3, 517--527.

\bibitem{Bord}
M.~Bordemann, \emph{Nondegenerate invariant bilinear forms on nonassociative
  algebras}, Acta Math. Univ. Comenian. (N.S.) \textbf{66} (1997), no.~2,
  151--201.

\bibitem{CJ}
Grant Cairns and Sebastian Jambor, \emph{The cohomology of the {H}eisenberg
  {L}ie algebras over fields of finite characteristic}, Proc. Amer. Math. Soc.
  \textbf{136} (2008), no.~11, 3803--3807.

\bibitem{max1}
A.~Caranti, S.~Mattarei, and M.~F. Newman, \emph{Graded {L}ie algebras of
  maximal class}, Trans. Amer. Math. Soc. \textbf{349} (1997), no.~10,
  4021--4051.

\bibitem{max2}
A.~Caranti and M.~F. Newman, \emph{Graded {L}ie algebras of maximal class.
  {II}}, J. Algebra \textbf{229} (2000), no.~2, 750--784.

\bibitem{Fial}
Alice Fialowski, \emph{On the classification of graded {L}ie algebras with two
  generators}, Moscow Univ. Math. Bull. \textbf{38} (1983), no.~2, 76--79.

\bibitem{FM}
Alice Fialowski and Dmitri Millionschikov, \emph{Cohomology of graded {L}ie
  algebras of maximal class}, J. Algebra \textbf{296} (2006), no.~1, 157--176.

\bibitem{goncharova}
L.~V. Gon{\v{c}}arova, \emph{Cohomology of {L}ie algebras of formal vector
  fields on the line}, Functional Anal. Appl. \textbf{7} (1973), no.~2, 91--97.

\bibitem{max3}
G.~Jurman, \emph{Graded {L}ie algebras of maximal class. {III}}, J. Algebra
  \textbf{284} (2005), no.~2, 435--461.

\bibitem{million}
Dmitri~V. Millionschikov, \emph{Graded filiform {L}ie algebras and symplectic
  nilmanifolds}, Geometry, topology, and mathematical physics, Amer. Math. Soc.
  Transl. Ser. 2, vol. 212, Amer. Math. Soc., Providence, RI, 2004,
  pp.~259--279.

\bibitem{SZ}
Aner Shalev and Efim~I. Zelmanov, \emph{Narrow {L}ie algebras: a coclass theory
  and a characterization of the {W}itt algebra}, J. Algebra \textbf{189}
  (1997), no.~2, 294--331.

\bibitem{skol}
Emil Sk{\"o}ldberg, \emph{The homology of {H}eisenberg {L}ie algebras over
  fields of characteristic two}, Math. Proc. R. Ir. Acad. \textbf{105A} (2005),
  no.~2, 47--49.

\bibitem{Ts}
Ioannis Tsartsaflis, \emph{On the {B}etti numbers of filiform {L}ie algebras
  over fields of characteristic two}, \url{http://arxiv.org/1511.03132}, 2015,
  [math.RA].

\bibitem{Vergne}
Mich{\`e}le Vergne, \emph{Cohomologie des alg\`ebres de {L}ie nilpotentes.
  {A}pplication \`a l'\'etude de la vari\'et\'e des alg\`ebres de {L}ie
  nilpotentes}, Bull. Soc. Math. France \textbf{98} (1970), 81--116.

\end{thebibliography}

\end{document}